\documentclass[twoside,11pt]{amsart}

\usepackage[english]{babel}
\usepackage{url}
\usepackage{graphicx}
\usepackage[colorinlistoftodos]{todonotes}
\usepackage{amsmath,amscd,amsthm,amstext,amsfonts,amssymb}
\usepackage{tikz-cd}
\usepackage{fancyhdr}
\usepackage{enumerate}
\usepackage{xcolor}

\usepackage[colorlinks=true,linkcolor=blue,citecolor=blue,urlcolor=blue]{hyperref}
\usepackage{cleveref}

\numberwithin{equation}{section}
\newtheorem{thm}[equation]{Theorem}
\newtheorem*{main}{Theorem}

\newtheorem{prop}[equation]{Proposition}
\newtheorem{cor}[equation]{Corollary}
\newtheorem{lem}[equation]{Lemma}

\theoremstyle{definition}

\newtheorem{qu}[equation]{Question}
\newtheorem{rem}[equation]{Remark}

\renewcommand{\dim}{\mathsf{dim}}
\renewcommand{\deg}{\mathsf{deg}}
\renewcommand{\bmod}{\,\mathsf{mod}\,}
\newcommand\ind{\operatorname{\mathsf{ind}}}

\newcommand\End{\operatorname{\mathsf{End}}}

\newcommand\id{\operatorname{\mathsf{id}}}

\newcommand\hh{\operatorname{\mathbb{H}}}
\newcommand{\car}{\mathsf{char}}
\newcommand{\can}{\operatorname{\mathsf{can}}}

\newcommand{\vf}{\varphi}
\newcommand{\mg}[1]{{#1}^{\times}}
\newcommand{\sq}[1]{{#1}^{\times 2}}
\newcommand{\scg}[1]{\mg{#1}/\sq{#1}}
\newcommand{\s}{\sigma}
\newcommand{\nat}{\mathbb{N}}

\newcommand{\la}{\langle}
\newcommand{\ra}{\rangle}
\newcommand{\lla}{\la\!\la}
\newcommand{\rra}{\ra\!\ra}

\renewcommand{\leq}{\leqslant}
\renewcommand{\geq}{\geqslant}

\newcommand\Ad{\operatorname{\mathsf{Ad}}}
\newcommand\ad{\operatorname{\mathsf{ad}}}

\newcommand{\mc}{\mathcal}
\newcommand{\N}{\mathsf{N}}

\newcommand{\dyn}[1]{\mathrm{#1}}
\newcommand{\disc}{\mathsf{disc}}
\newcommand{\Sim}{{\bf\mathsf{Sim}}}
\newcommand{\PSim}{{\bf\mathsf{PSim}}}
\newcommand{\G}{\mathsf{G}}

\newcommand{\C}{\mathsf{C}}

\newcommand{\Hyp}{\mathsf{Hyp}}

\newcommand{\D}{\mathsf{D}}

\newcommand{\wi}{\mathsf{i}}

\newcommand{\cc}{\mathbb{C}}

\newcommand{\wt}{\widetilde}

\title{Totally decomposable algebras with involution and \texorpdfstring{$R$}{R}-triviality}

\date{18 August, 2026}

\author{M.~Archita}
\author{Karim Johannes Becher}

\address{University of Antwerp, Department of Mathematics, Antwerp, Belgium.}
\email{karimjohannes.becher@uantwerpen.be}
\email{archita.mondal@uantwerpen.be}

\begin{document}

\begin{abstract}
We show that the group of proper projective similitudes of a totally decomposable algebra with orthogonal or symplectic involution of index at most $2$ over a field of characteristic different from $2$ is $R$-trivial.
We further give an example showing that the statement does not extend to index $4$ in the orthogonal case.


\medskip\noindent
{\sc Classification (MSC 2020):} 11E04, % Quadratic forms over general fields 
11E57, % Classical groups
11E81, % Algebraic theory of quadratic forms; Witt groups and rings
14E08, % Rationality questions in algebraic geometry
20G15 % Linear algebraic groups over arbitrary fields
\end{abstract}

\maketitle

\section{Introduction}

It had been conjectured in \cite[\S 7.2, Conjecture]{PR94} that semisimple adjoint linear algebraic groups are rational.
This turned out to be correct for groups of type ${}^1\!\dyn{A}$ and $\dyn{B}$, but wrong in general for adjoint semisimple groups of  types ${}^2\!\dyn{A}$, $\dyn{C}$ and $\dyn{D}$.

By Weil's classification \cite{Wei61}, any absolutely simple adjoint group is given as the group of proper projective similitudes of some suitable central simple algebra with involution or quadratic pair; see also \cite[\S~26]{KMRT98}.

For an algebraic group $\mc{G}$ over a field $K$, the set of $R$-equivalence classes $\mc{G}(F)/R$ for an arbitrary field extension $F/K$ has a group structure, and it provides an obstruction to $\mc{G}$ being  stably rational.
The group $\mc{G}$ is called \emph{$R$-trivial} if the group $\mc{G}(F)/R$ is trivial for every field extension $F/K$.
By \cite[Prop.~1]{Mer96}, every stably rational group is $R$-trivial.

Let $K$ be a field of characteristic different from $2$.
Let $(A,\s)$ be a \emph{$K$-algebra with involution}, that is, $A$ is a central simple $K$-algebra and $\s$ is a $K$-linear involution on $A$. (Hence we restrict our consideration to involutions of the first kind.)
Let $\mathbf{PSim}(A,\s)$ denote the algebraic group of  projective similitudes of $(A,\s)$ and $\mathbf{PSim}^+(A,\s)$ the connected component of the identity.
The latter is a classical absolutely simple adjoint linear algebraic group, which is of type $\dyn{C}$ if $\s$ is symplectic, respectively of type $\dyn{B}$ or $\dyn{D}$ if $\s$ is orthogonal, depending on whether the degree of $A$ is odd or even.
In \cite{Mer96}, Merkurjev investigated the problem of deciding whether $\mathbf{PSim}^+(A,\s)$ is rational, stably rational or $R$-trivial. 
For a field extension $F/K$, \cite[Theorem 1]{Mer96} identifies the group of $R$-equivalence classes $\mathbf{PSim}^+(A,\s)(F)/R$ with a quotient of two subgroups of $\mg{F}$ related to $(A,\s)$.
We recall this characterisation in \Cref{Mer:T1}.
On the one hand, this is the basis of various constructions of adjoint groups which are not stably rational.
On the other hand, it can be used to establish sufficient conditions on $(A,\s)$ for the $R$-triviality of $\mathbf{PSim}^+(A,\s)$.

In this article, we prove that one such case is given by totally decomposable algebras with orthogonal or symplectic involution of index at most $2$.
We call $(A,\s)$ \emph{totally decomposable} 
if $$(A,\s)\simeq \Ad(\psi)\otimes\left(\bigotimes_{i=1}^r (Q_i,\tau_i)\right)$$
for some $r\in\nat$, $K$-quaternion algebras with involution $(Q_i,\tau_i)$ for $1\leq i\leq r$ and the adjoint algebra with involution $\Ad(\psi)$ (whose definition we recall in \Cref{S:simfact}) of some odd-dimensional quadratic form $\psi$ over $K$. (Taking $\psi=\la 1\ra$ corresponds to omitting the factor $\Ad(\psi)$, and sometimes, this more restrictive condition is taken for total decomposability.)

If $\s$ is orthogonal and $A$ is split, then we obtain by \cite[Theorem 1]{Bec08} that $(A,\s)\simeq \Ad(\psi\otimes \pi)$
for some Pfister form $\pi$ over $K$, and in this case it follows by \cite[Prop.~7]{Mer96} that $\mathbf{PSim}^+(A,\s)$ is stably rational, and in particular $R$-trivial.
One may ask whether these conclusions still hold without assuming that $A$ is split.
In this article, we show:

\begin{main}[\Cref{main}]
Let $(A,\sigma)$ be a totally decomposable $K$-algebra with involution of the first kind such that $\ind A\leq 2$. Then the linear algebraic group $\mathbf{PSim}^+(A,\s)$ is $R$-trivial. 
\end{main}

In the case where $A$ is split or $\s$ is symplectic, this is a direct consequence of the Pfister Factor Conjecture proven in \cite{Bec08}.
The interesting case is when $\s$ is orthogonal and $\ind A=2$.
% For the proof, we crucially rely on the characterization of $R$-equivalence in $\mathbf{PSim}^+(A,\s)$ from \cite[Theorem 1]{Mer96}. 

It turns out that the hypothesis that $\ind A\leq 2$ is necessary for the conclusion in \Cref{main}, at least when $\s$ is orthogonal.
An example to show this is obtained with \Cref{C:ind4counterex} over the field $K=\cc(X,Y,Z)$ by the  totally decomposable  
$K$-algebra with orthogonal involution  $$(A,\s)=\Ad\lla X\rra\otimes(Q_1,\can_{Q_1})\otimes (Q_2,\can_{Q_2})$$
where $Q_1=(Y,Z)_K$, $Q_2=((X-1)Y,(X+1)Z)_K$,
and $\can_{Q_i}$ denotes the canonical involution on $Q_i$ for $i\in\{1,2\}$.
Here $\mathbf{PSim}^+(A,\s)(K)/R=\{1\}$, but there exists a field extension $F/K$ such that $\mathbf{PSim}^+(A,\s)(F)/R\neq \{1\}$, whereby $\mathbf{PSim}^+(A,\s)$ is not $R$-trivial.

In light of the recent result of Biswas \cite{Bis26}, this also provides another case where the so-called \emph{Weak Approximation} fails: for some field extension $L/K$ and a discrete valuation $v$ with corresponding completion $L^v$, the $L$-rational points $\mathbf{PSim}^+(A,\s)(L)$ are not dense in $\mathbf{PSim}^+(A,\s)(L^v)$.

\section{Similarity factors}
\label{S:simfact}

We use standard notation and terminology from quadratic form theory, taking \cite{EKM08} and \cite{Lam05} as main references.
For basic concepts on central simple algebras and their involutions, we refer to \cite[Chap.~8]{Scha85} and \cite{KMRT98}.
By a \emph{quadratic form over $K$} we always mean a regular quadratic form, given by a quadratic map  on a finite-dimensional $K$-vector space.

Given a quadratic form $\vf$ over $K$, defined on the $K$-vector space $V_\vf$,
we associate the symmetric 
$K$-bilinear form $$b_\vf:V_\vf\times V_\vf\to K,\,(x,y)\mapsto \vf(x+y)-\vf(x)-\vf(y)$$ and its 
\emph{adjoint algebra with involution} $$\Ad(\vf)=(\End(V_\vf),\ad_\vf).$$
Here $\ad_\vf:\End_K(V_\vf)\to\End_K(V_\vf)$ denotes the $K$-linear involution determined by
$$ b_\vf(x,f(y))=b_\vf(\ad_\vf(f)(x),y)\text{ for every } x,y\in V_\vf\text{ and } f\in\End_K(V_\vf).$$
Note that $\Ad(\vf)$ is a split $K$-algebra with orthogonal involution.

Let $(A,\s)$ be a $K$-algebra with involution. We call $\s$ \emph{hyperbolic} if there exists an element $e\in A$ such that $e^2=e$ and $\s(e)=1-e$. 
We call
$(A,\s)$ \emph{hyperbolic} if $\s$ is hyperbolic.
This can only occur when $\deg A$ is a multiple of $2\ind A$, so in particular $\deg A$ must be even.
We call $(A,\s)$ \emph{split hyperbolic} if $A$ is split and $\s$ is hyperbolic.
For a field extension $L/K$, we obtain an $L$-algebra with involution $(A,\s)_L=(A_L,\s_L)$, where $A_L=A\otimes_KL$ and $\s_L=\s\otimes\id_L$.

We set $\Sim(A,\s)=\{x\in A\mid \s(x)x\in \mg{K}\}$, which is a subgroup of $\mg{A}$. 
We 
denote by $\PSim(A,\s)$ the quotient group $\Sim(A,\s)/\mg{K}$.
The elements of $\Sim(A,\s)$ are called \emph{similitudes of $(A,\s)$}, and those of $\PSim(A,\s)$ are called \emph{projective similitudes of $(A,\s)$}.

Letting 
$$\mathbf{PSim}(A,\s)(F)=\PSim(A_{F},\s_{F})$$
for every extension $F/K$ defines a linear algebraic group over $K$. We denote by $\mathbf{PSim}^+ (A, \sigma)$ the connected component of the identity in 
$\mathbf{PSim} (A,\sigma)$. 
This is a classical absolutely simple adjoint linear algebraic group.
Note that $\mathbf{PSim}^{+}(A,\s)=\mathbf{PSim}(A,\s)$ holds unless $\s$ is orthogonal and $\deg A$ is even.
In any case, $\mathbf{PSim}^+(A,\s)(F)$ has index at most $2$ in $\mathbf{PSim} (A,\s)(F)$ for every field extension $F/K$; see \cite[Prop. 12.23]{KMRT98}.
We obtain a natural group homomorphism 
$$\mu:\Sim(A,\s)\to \mg{K}, x\mapsto \s(x)x.$$
Its image is denoted by $\G(A,\s)$. In other terms,
 $$\G(A,\s)=\{\s(x)x\mid x\in\mg{A}\}\cap\mg{K}.$$
 The elements of this subgroup of $\mg{K}$ are called \emph{similarity factors} or \emph{multipliers of similitude of $(A,\s)$}.
We further set
$$\G^+(A,\s)=\{\mu(a)\mid a\mg{K}\in\mathbf{PSim}^+(A,\s)(K)\}.$$
This is a subgroup of $\G(A,\s)$ of index at most $2$.
We refer to \cite[\S12]{KMRT98} for a discussion of similitudes and their multipliers.

We denote by $\Hyp(A,\s)$ the subgroup of $\mg{K}$ generated by the nonzero elements of $K$ which are norms from some finite field extension $L/K$ such that $(A,\s)_L$ becomes hyperbolic.

\begin{thm}[Merkurjev]
\label{Mer:T1} 
For any field extension $F/K$, $\sq{F}\Hyp(A_{F},\s_{F})$ is a subgroup of $\G^+(A_{F},\s_{F})$, and $\mu$ induces a group isomorphism 
$$\mathbf{PSim}^+(A,\s)(F)/R\stackrel{\simeq}{\longrightarrow} \G^+(A_{F},\s_{F})/\sq{F}\Hyp(A_{F},\s_{F}).$$
\end{thm}

\begin{proof}
    See \cite[Theorem 1]{Mer96}.
\end{proof}

For $a\in\mg{K}$, we write $\lla a\rra$ for the quadratic form $\la 1,-a\ra$, which is the norm form of the quadratic \'etale extension $K[T]/(T^2-a)$ of $K$.

\begin{prop}\label{sim-fac-hyp}
    We have 
    $$\G(A,\s)=\{a\in\mg{K}\mid \Ad\lla a\rra\otimes (A,\s)\text{ is hyperbolic}\}.$$
\end{prop}
\begin{proof}
    See \cite[Prop.~12.20]{KMRT98}.
\end{proof}

\begin{prop}\label{odddimnonzerodiv}
Let $\psi$ be an odd-dimensional quadratic form over $K$.
Then $(A,\s)$ is hyperbolic if and only if $\Ad(\psi)\otimes(A,\s)$ is hyperbolic. 
Furthermore, $$\G(\Ad(\psi)\otimes (A,\s))=\G(A,\s).$$
\end{prop}
\begin{proof}
If $(A,\s)$ is hyperbolic, then so is $\Ad(\psi)\otimes (A,\s)$.
Assume now that $\Ad(\psi)\otimes (A,\s)$ is hyperbolic.
If $\psi=m\times \la 1\ra$ for an odd number $m\in\nat$, then the desired conclusion that $(A,\s)$ is hyperbolic follows from Scharlau's result \cite[Theorem 5.1]{Scha70} that the torsion part of the Witt group of hermitian forms over $(A,\s)$ is $2$-primary; see also \cite[Theorem 6.5]{BU18}.
In the general case, fix $k\in\nat$ such that $\dim(\psi)=2k+1$.
By Lewis' Theorem \cite[Chap.~VIII, Theorem 8.13]{Lam05}, the class of $\psi$ in the Witt ring of $K$ is a zero of the polynomial $\prod_{i=0}^k(X^2-(2i+1)^2)$. This implies that for $m=\prod_{i=0}^k (2i+1)^2$ the quadratic form $m\times \la 1\ra$ is Witt equivalent to a multiple of $\psi\otimes\psi$, and hence of $\psi$.
Since $\Ad(\psi)\otimes (A,\s)$ is hyperbolic, we therefore obtain that $\Ad(m\times \la 1\ra)\otimes (A,\s)$ is hyperbolic.
Since $m$ is odd, it follows by the first case that $(A,\s)$ is hyperbolic.
    This shows the first part of the statement.

    For $c\in\mg{K}$, we have $$\Ad\lla c\rra\otimes(\Ad(\psi)\otimes (A,\s))\simeq \Ad(\psi)\otimes (\Ad\lla c\rra\otimes (A,\s)),$$ and by the first part, this algebra with involution is hyperbolic if and only if
    $\Ad\lla c\rra\otimes (A,\s)$ is hyperbolic. 
    Now the last part follows by \Cref{sim-fac-hyp}.
\end{proof}

\begin{rem}
    Let $\psi$ be a quadratic form of odd dimension over $K$.
    By \Cref{odddimnonzerodiv} and \Cref{Mer:T1}, it follows that $\mathbf{PSim}^+(\Ad(\psi)\otimes(A,\s))$ is $R$-trivial if and only if $\mathbf{PSim}^+(A,\s)$ is $R$-trivial.
    More generally, applying \cite[Prop.~4.5]{Mer98}, one sees that $\mathbf{PSim}^+(\Ad(\psi)\otimes(A,\s))\times \mathbb{A}^m$ and $\mathbf{PSim}^+(A,\s)\times \mathbb{A}^{m'}$ are birationally isomorphic for certain $m,m'\in\nat$.
\end{rem}

For a quadratic form $\vf$ over $K$ of dimension at least $3$, the function field of the associated projective quadric is denoted by $K(\vf)$. 
We refer to \cite[Section 22]{EKM08} for a discussion of function fields.

\begin{prop}\label{simsplitreduction}
Let $Q$ be a $K$-quaternion algebra, $\nu$ its norm form and set $L=K(\nu)$.
Let $(A,\s)$ be a $K$-algebra with orthogonal involution such that $A$ is Brauer equivalent to $Q$.
Then $\s_L$ is hyperbolic if and only if $\s$ is hyperbolic.
Furthermore, $A_{L}$ is split and $\G(A,\s)=\G((A,\s)_{L})\cap \mg{K}$.
\end{prop}
\begin{proof}
    Clearly, $A_L$ is split.
    We set $L_0=K(\nu')$ where $\nu'$ is the subform satisfying $\nu\simeq \la 1\ra\perp\nu'$.
    Then $L/L_0$ is a rational function field in one variable.
    Hence, the statement on hyperbolicity follows from the corresponding statement with $L_0$ in the place of $L$, which is Dejaiffe's Theorem \cite{Dej01}.
    Applying this now to $\Ad\lla c\rra\otimes (A,\s)$ in the place of $(A,\s)$ for arbitrary $c\in\mg{K}$, we conclude by \Cref{sim-fac-hyp} that $\G(A,\s)=\G((A,\s)_{L})\cap \mg{K}$.
\end{proof}

For a quadratic form $\vf$ over $K$, we will use the abbreviations $\G(\vf), \G^+(\vf)$ and $\Hyp(\vf)$ when referring to the corresponding groups for the adjoint algebra with involution $\Ad(\vf)$.
We observe that 
\[\G(\vf)=\{a\in\mg{K}\mid a\vf\simeq \vf\}=\G^+(\vf).\]

\begin{prop}\label{Jacobson}
    Let $Q$ be a $K$-quaternion algebra, let $\gamma$ be the canonical involution on $Q$ and $\nu$ the norm form of $Q$.
    Then for every $K$-algebra with symplectic involution $(A,\s)$ where $A$ is Brauer equivalent to $Q$, there exists a quadratic form $\vf$ over $K$ such that 
    $$(A,\s)\simeq \Ad(\vf)\otimes (Q,\gamma).$$
    Given such a presentation, $\s$ is isotropic, respectively~hyperbolic, if and only if $\vf\otimes\nu$ has the corresponding property, and 
    furthermore $\G(A,\s)=\G(\vf\otimes\nu)$.
\end{prop}
\begin{proof}
    This follows from Jacobson's Theorem \cite{Jac40}; see also \cite[Theorem 10.1.1]{Scha85}.
\end{proof}

\section{Pfister forms}
\label{S:Pfister}

For a quadratic form $\vf$ over $K$, we set $\G(\vf)=\{a\in\mg{K}\mid a\vf\simeq\vf\}$ and observe that $\G(\vf)=\G(\Ad(\vf))$, and we further denote by $\D(\vf)$ the set of elements of $\mg{K}$ which are represented by $\vf$ over $K$.
Recall that $\D(\vf)=\G(\vf)$ holds whenever $\vf$ is a Pfister form; see e.g.~\cite[Chap.~X, Theorem 1.8]{Lam05}.

We call a Pfister form \emph{nontrivial} if it is not isometric to $\la 1\ra$.

\begin{prop}\label{Pfisterform}
    Let $\pi$ be a nontrivial Pfister form over $K$.
    For $c\in\D(\pi)$, there exists $d\in\mg{K}$ such that $c\in\D\lla d\rra$ and $\pi_{K(\sqrt{d})}$ is hyperbolic.
\end{prop}
\begin{proof}
Let $c\in\D(\pi)$.  
Then $c\in\D\lla d\rra$ for some $d\in\mg{K}$ such that $\lla d\rra$ is a subform of $\pi$.
Then $\pi_{K(\sqrt{d})}$ is isotropic, and since $\pi$ is a Pfister form, it follows that $\pi_{K(\sqrt{d})}$ is hyperbolic. 
\end{proof}

We denote by $\hh_K$ the hyperbolic plane over $K$, and for $d\in\nat$, we write $d\times\hh_K$ for the $d$-fold orthogonal sum $\hh_K\perp\ldots\perp\hh_K$, 
which is up to isometry the unique $2d$-dimensional hyperbolic form over $K$.
Given a quadratic form $\vf$ over $K$, we denote by $\wi(\vf)$ its Witt index. Witt's Cancellation Theorem \cite[Theorem 8.4]{EKM08}
implies the following subform criterion.

\begin{prop}\label{subform}
    Let $\psi$ and $\vf$ be quadratic forms over $K$.
    Then $\psi$ is a subform of $\vf$ if and only if $\wi(\vf\perp-\psi)\geq \dim(\psi)$.
\end{prop}
\begin{proof}
    Let $d=\dim(\psi)$. Note that $\psi\perp-\psi\simeq d\times\hh_K$.
    In view of Witt's Cancellation Theorem \cite[Theorem 8.4]{EKM08}, we obtain for an arbitrary quadratic form $\rho$ over $K$ that 
    $\vf\simeq \psi\perp\rho$ holds if and only if $\vf\perp-\psi\simeq \rho\perp d\times\hh_K$.
    The existence of such a quadratic form $\rho$ is therefore equivalent to having that $\wi(\vf\perp-\psi)\geq d$.
\end{proof}

We formulate a variant of the so-called Subform Theorem \cite[Theorem 22.5]{EKM08}.

\begin{prop}\label{Pfister-subform}
    Let $\vf$ and $\psi$ be Pfister forms with $\dim(\vf)>\dim(\psi)\geq 2$.
    Then $\vf_{K(\psi)}$ is hyperbolic if and only if $\psi$ is a subform of $\vf$.
\end{prop}
\begin{proof}
    If $\vf$ is anisotropic and $\vf_{K(\psi)}$ is hyperbolic, then it follows by \cite[Theorem 22.5]{EKM08} that $\psi$ is a subform of $\vf$.
    If $\vf$ is isotropic, then it is hyperbolic, and since $\dim(\vf)\geq 2\dim(\psi)$, it follows that $\psi$ is a subform of $\vf$.
    Conversely, if $\psi$ is a subform of $\vf$, then $\vf_{K(\psi)}$ is isotropic and hence hyperbolic.
\end{proof}

\begin{prop}\label{Pfister-multiple}
    Let $\alpha$ be an anisotropic Pfister form and $\varphi$ a quadratic form over $K$.
    Then there exists a quadratic form $\vf'$ such that 
    $\alpha\otimes \vf'$ is anisotropic and Witt equivalent to  $\alpha\otimes \vf$ and $\dim(\vf')\equiv \dim(\vf)\,\bmod 2$.
    In particular, if $\dim(\vf)$ is even, then $\wi(\alpha\otimes \vf)$ is a multiple of $\dim(\alpha)$.
\end{prop}
\begin{proof}
    If $\alpha\simeq \la 1\ra$ the statement is trivial, so we may assume that $\alpha$ is a nontrivial Pfister form and consider its function field $K(\alpha)$.
    Then $\alpha_{K(\alpha)}$ is hyperbolic.
    Let $\vartheta$ be the anisotropic part of $\alpha\otimes\vf$. It follows  by \cite[Chap.~X, Theorem 4.11]{Lam05} that $\vartheta\simeq \alpha\otimes\vf'$ for some quadratic form $\vf'$. 
    Then $(\varphi\perp-\varphi')\otimes \alpha$ is hyperbolic. 
    As $\alpha$ is not hyperbolic, it follows by \cite[Chap.~VIII, Cor.~8.5]{Lam05} (or by \Cref{odddimnonzerodiv}) that $\dim(\varphi\perp-\varphi')$ is even, whereby $\dim (\vf)\equiv \dim (\vf')\bmod 2$.
\end{proof}

We now apply the above subform criteria to compute the similarity factors of a Pfister form after base change to the function field of a related Pfister form.

\begin{lem}\label{L:Pfister-fufi-G}
    Let $\alpha$ and $\rho$ be Pfister forms over $K$ such that $\rho$ is nontrivial. Let $b\in\mg{K}$. 
    Set $\psi=\alpha\otimes \lla b\rra$, $\pi=\alpha\otimes\rho$ and $\wt{\pi}=\alpha\otimes (\la b\ra\perp\rho')$.
    Then \[\G(\pi_{K(\psi)})\cap\mg{K}=\D(\pi)\cdot\D(\wt{\pi}).\]
\end{lem}
\begin{proof}
    If $\pi$ is hyperbolic, then $\G(\pi_{K(\psi)})\cap\mg{K}=\mg{K}=\D(\pi)=\D(\pi)\cdot\D(\wt{\pi})$.
    We may thus assume that $\pi$ is anisotropic. In particular $\alpha$ is anisotropic.
    
    Consider $c\in\mg{K}$.
    We have $c\in\G(\pi_{K(\psi)})$ if and only if $(\lla c\rra\otimes\pi)_{K(\psi)}$ is hyperbolic. 
    In view of \Cref{Pfister-subform}, this is if and only if $\psi$ is a subform of $\lla c\rra\otimes\pi$.
    Note that $\psi=\alpha \perp -b\alpha$ and $\lla c\rra\otimes \pi =\alpha\perp (\alpha\otimes \rho') \perp - c\pi$.
    By Witt's Cancellation Theorem \cite[Theorem 8.4]{EKM08}, $\psi$ is a subform of $\lla c\rra\otimes\pi$ if and only if $-b\alpha$ is a subform of $(\alpha\otimes \rho') \perp - c\pi$. 
    Now observe that
    $$b\alpha\perp (\alpha\otimes \rho')\perp -c\pi\simeq \alpha\otimes (\la b\ra\perp\rho'\perp -c\rho) \simeq \wt{\pi}\perp-c\pi.$$     
    It follows by \Cref{Pfister-multiple} that $\wi(\wt{\pi}\perp-c\pi)$ is a multiple of $\dim(\alpha)$.
    Therefore $\wt{\pi}\perp-c\pi$ is isotropic if and only if $\wi(b\alpha\perp (\alpha\otimes \rho')\perp -c\pi)\geq \dim(b\alpha)$, and by \Cref{subform} this is if and only if $-b\alpha$ is a subform of $(\alpha\otimes \rho')\perp -c\pi$.
    
    Hence, we have shown that $c\in\G(\pi_{K(\psi)})$ if and only if $\wt{\pi}\perp-c\pi$ is isotropic. The latter is equivalent to having that $c\in\D(\pi)\cdot \D(\wt{\pi})$. 
\end{proof}

For an orthogonal involution $\tau$ on a central simple $K$-algebra, we denote by $\disc(\tau)$ its discriminant, viewed as a class in $\scg{K}$.

\begin{lem}\label{ad}
Let $a,b\in\mg{K}$, $Q=(a,b)_K$ and $\tau$ an orthogonal involution on $Q$ with $\disc(\tau)=a\sq{K}$.
Let $\rho$ be a Pfister form over $K$, $\psi$ an odd-dimensional quadratic form over $K$ and $$(A,\s)=\Ad(\psi\otimes\rho)\otimes (Q,\tau).$$
Let $\pi=\lla a\rra\otimes \rho$ and $\wt{\pi}=\lla a\rra\otimes (\la b\ra\perp\rho')$.
Then $$\G(A,\s)=\D(\pi)\cdot \D(\wt{\pi}).$$
    Furthermore, for any field extension $K'/K$, the form $\wt{\pi}_{K'}$ is isotropic if and only if $\s_{K'}$ is hyperbolic.
    \end{lem}
\begin{proof}
Let $\nu= \lla a,b\rra$. Set $L=K(\nu)$.
Then $Q_{L}$ and $A_{L}$ are split, and 
\[(A,\s)_{L}\simeq \Ad(\psi_L)\otimes\Ad(\rho_{L})\otimes\Ad(\lla a\rra_{L})\simeq \Ad((\psi\otimes\pi)_{L}).\]

By \Cref{simsplitreduction} and \Cref{odddimnonzerodiv}, we have
\[\G(A,\s)=\G((\psi\otimes \pi)_L)\cap \mg{K}\quad\text{ and 
}\quad\G((\psi\otimes\pi)_{L})=\G(\pi_{L}).\]
Using \Cref{L:Pfister-fufi-G}, we therefore conclude that 
\[\G(A,\s)=\G(\pi_{L})\cap\mg{K}=\D(\pi)\cdot \D(\wt{\pi}).\]

Consider now a field extension $K'/K$.
Note that 
\[\wt{\pi}_{K'(\nu)}\simeq \pi_{K'(\nu)}\qquad\text{ and }\qquad(A,\s)_{K'(\nu)}\simeq \Ad((\psi\otimes\pi)_{K'(\nu)}).\]
Assume first that $\wt{\pi}_{K'}$ is isotropic.
Then $\pi_{K'(\nu)}$ is isotropic and thus hyperbolic, because it is a Pfister form. 
It follows that $\s_{K'(\nu)}$ is hyperbolic, and hence $\s_{K'}$ is hyperbolic, by \Cref{simsplitreduction}.
Assume conversely that $\s_{K'}$ is hyperbolic.
Then $\pi_{K'(\nu)}$ is hyperbolic, and it follows by \Cref{Pfister-subform} that $\nu_{K'}$ is a subform of $\pi_{K'}$.
Since $\wt{\pi}$ is Witt equivalent to $\pi\perp-\nu$ and $\dim(\wt{\pi})=\dim(\pi)$, we conclude that $\wt{\pi}_{K'}$ is isotropic.
\end{proof}

\section{\texorpdfstring{$R$-triviality in index at most $2$}{R-triviality in index at most 2}}
\label{S:main-result}

When $(A,\sigma) \simeq \Ad(\psi \otimes \pi)$ for a Pfister form $\pi$ and an odd-dimensional quadratic form $\psi$ over $K$, then we have $\G(A,\sigma) = \G(\pi) = \Hyp(\pi) = \Hyp(A,\sigma)$. In this case, \Cref{Pfisterform} yields for any given element $a \in \G(A,\sigma)$ a single quadratic extension $L/K$ such that $\sigma_L$ is hyperbolic and $a \in \N_{L/K}(\mg{L})$.

The goal of this section is to extend this phenomenon beyond the split case. 
Specifically, we show that $\G(A,\sigma) = \Hyp(A,\sigma)$ holds for all totally decomposable $K$-algebras with involution of index at most $2$. The key strategy is to use an auxiliary Pfister form $\pi$ associated with $(A,\sigma)$ to factor any multiplier of similitude into a product of two elements, each occurring as a norm from a quadratic extension over which $\s$ becomes hyperbolic.

\begin{thm}\label{T:main}
    Let $(A,\sigma)$ be a totally decomposable $K$-algebra with  involution of the first kind of even degree and with $\ind A\leq 2$. 
    For every $c\in\G(A,\s)$ there exist $d_1,d_2\in\mg{K}$ such that $\s_{K(\sqrt{d_i})}$ is hyperbolic for $i\in\{1,2\}$ and $c\in\D\lla d_1\rra\cdot \D\lla d_2\rra$.
\end{thm}
\begin{proof}
It follows from the hypothesis, by \cite[Corollary]{Bec08} if $\s$ is symplectic and by \cite[Theorem 2]{Bec08} if $\s$ is orthogonal, that there exists a decomposition
$$(A,\s)\simeq \Ad(\psi\otimes \rho)\otimes (Q,\gamma)$$
with a $K$-quaternion algebra with involution of the first kind $(Q,\gamma)$,  a Pfister form $\rho$ over $K$ and an odd-dimensional quadratic form $\psi$ over $K$.

Assume first that $\s$ is symplectic. Then $\gamma$ is symplectic and hence equal to the canonical involution on $Q$.
Let $\nu$ denote the norm form of $Q$.
This is a $2$-fold Pfister form.
By \Cref{Jacobson} and \Cref{odddimnonzerodiv} we have that  
$\G(A,\s)=\G(\psi\otimes\rho\otimes\nu)=\G(\rho\otimes\nu)$ and,
for any field extension $L/K$, $(A,\s)_L$ is hyperbolic if and only if $(\psi\otimes \rho\otimes \nu)_L$ is hyperbolic, if and only if  $(\rho\otimes\nu)_L$ is hyperbolic.
Consider an element $c\in\G(A,\s)$.
Since $\G(A,\s)=\G(\rho\otimes \nu)$ and $\rho\otimes \nu$ is a Pfister form, we obtain by \Cref{Pfisterform} that 
$c\in \D\lla d\rra$ for some $d\in\mg{K}$ such that $(\nu\otimes\rho)_{K(\sqrt{d})}$ is hyperbolic.
Hence, we may take $d_1=d_2=d$
to have that $c=c\cdot 1\in\D\lla d_1\rra\cdot\D\lla d_2\rra$ 
and $\s_{K(\sqrt{d_i})}$ is hyperbolic for $i\in\{1,2\}$.

Assume now that $\s$ is orthogonal.
If $\deg A\equiv 2\bmod 4$, then $\rho\simeq \la 1\ra$ and we fix $d\in\mg{K}$ such that $\disc(\gamma)=d\sq{K}$ to have that $\G(A,\s)=\D\lla d\rra$ and $\s_{K(\sqrt{d})}$ is hyperbolic, so the claim holds trivially with $d_1=d_2=d$.
We may now assume that $\deg A$ is a multiple of $4$.
We choose $a,b\in\mg{K}$ such that $\disc(\gamma)=a\sq{K}$ and $Q\simeq (a,b)_K$.
Consider $c\in \G(A,\s)$. 
Let $\pi$ and $\wt{\pi}$ be defined as in \Cref{ad}. 
Note that $\rho'$ is a subform of both $\pi$ and $\wt{\pi}$.
Using \Cref{ad} we can decompose $c=c_1\cdot c_2$ with certain $c_1\in \D(\pi)$ and $c_2\in \D(\wt{\pi})$. 
We now use this factorization of $c$ to construct two suitable quadratic extensions where $\s$ becomes hyperbolic.
Since $\dim(\rho')=\frac{1}2\deg A-1\geq 1$, we may pick an element $e\in \D(\rho')$. Since $e, c_1\in \D(\pi)=\G(\pi)$, we obtain that $e^{-1}c_1\in \D(\pi)$. 
By \Cref{Pfisterform}, there exists some $d_1\in\mg{K}$ such that $e^{-1}c_1\in \D\lla d_1\rra$ and $\pi_{K(\sqrt{d_1})}$ is hyperbolic. Thus $\s_{K(\sqrt{d_1})}$ is hyperbolic.
Observe further that $1,ec_2\in \D(e\wt{\pi})$.
Hence $e\wt{\pi}$ has a binary subform $\beta$ with $1,ec_2\in\D(\beta)$. 
Then $\beta\simeq \lla d_2\rra$ for some $d_2\in\mg{K}$.
We obtain that $ec_2\in\D\lla d_2\rra$ and $\wt{\pi}_{K(\sqrt{d_2})}$ is isotropic, whereby $\s_{K(\sqrt{d_2})}$ is hyperbolic, by \Cref{ad}.
We thus have that $c=e^{-1}c_1(ec_2)\in\D\lla d_1\rra\cdot \D\lla d_2\rra$ and $\s_{K(\sqrt{d_i})}$ is hyperbolic for $i\in\{1,2\}$, as desired.
\end{proof}

\begin{cor}\label{main}   
Let $(A,\sigma)$ be a totally decomposable $K$-algebra with involution of the first kind such that $\ind A\leq 2$. Then $\mathbf{PSim}^+(A,\s)$ is $R$-trivial. 
\end{cor}
\begin{proof}
Using \Cref{Mer:T1}, this follows immediately from \Cref{T:main}.
\end{proof}

\begin{qu}
Under the hypotheses of \Cref{main}, is  $\mathbf{PSim}^+(A,\s)$ stably rational?
\end{qu}

\begin{rem}
When $A$ is split or $\s$ is symplectic, a positive answer to this question readily follows from \cite[Prop.~7]{Mer96}. 
Since $R$-triviality is a necessary condition for stable rationality, \Cref{main} provides suggestive evidence toward stable rationality in the remaining case where $\s$ is orthogonal and $\ind A=2$.
\end{rem}

\section{An example of non-\texorpdfstring{$R$}{R}-triviality in degree \texorpdfstring{$8$}{8}}
\label{S:examples}

By \cite[\S42]{KMRT98}, any $K$-algebra with orthogonal involution $(A,\s)$ of degree $8$ with trivial discriminant occurs in a trialitarian triple, and $(A,\s)$ is  totally decomposable if and only if it occurs in a triple where one component is split and hence given by a quadratic form.
On the other hand, by \cite[Prop.~42.5]{KMRT98}, the groups of proper projective similitudes of the algebras of a trialitarian triple are isomorphic as algebraic groups.
Using this, one obtains examples of totally decomposable algebras with orthogonal involution of degree $8$ whose group of proper projective similitudes are not $R$-trivial.

We provide an explicit description of the decomposable algebra with orthogonal involution of degree $8$ obtained from a diagonalized quadratic form.
For a $K$-quaternion algebra $Q$, we denote by $\can_Q$ its canonical involution.

\begin{prop}\label{8dim-even-Clifford-computation}
Let $\vf$ be an $8$-dimensional quadratic form over $K$ and let $\tau$ denote the canonical involution on $\C_0(\vf)$.
%Furthermore, in this case $\mathbf{PSim}^+(\vf) \simeq  \mathbf{PSim}^+(A,\s)$. 
Let $c_1,\dots,c_8\in\mg{K}$ be such that $\vf\simeq \la c_1,\dots,c_8\ra$ and $d=c_1\cdots c_8$, whereby $\disc(\vf)=d\sq{K}$.
Let $Q_1=(-c_2c_4,-c_3c_4)_K$, $Q_2=(-c_6c_8,-c_7c_8)_K$, $Q_3=(c_1c_2c_3c_4,-c_1c_5)_K$ and let $\gamma$ be an orthogonal involution on $Q_3$ with $\disc(\gamma)=-c_1c_5\sq{K}$.
Set 
\[(A,\s)= (Q_1,\can_{Q_1})\otimes(Q_2,\can_{Q_2})\otimes (Q_3,\gamma).\] 
\begin{enumerate}[$(i)$]
    \item If $d\notin\sq{K}$, then 
\[(\C_0(\vf),\tau)\simeq (A,\s)_{K(\sqrt{d})}.\]
\item
If $d\in\sq{K}$, then 
\[(\C_0(\vf),\tau)\simeq (A,\s)\times (A,\s)\quad\text{ and }\quad\mathbf{PSim}^+(\vf) \simeq  \mathbf{PSim}^+(A,\s).\]
\end{enumerate}
\end{prop}
\begin{proof}
We fix an orthogonal basis $(e_1,\ldots,e_8)$ for the quadratic form $\vf$ such that $e_i^2=c_i$ for $1\leq i\leq 8$.
The basis elements $e_1,\dots,e_8$ generate the Clifford algebra $\C(\vf)$, while the products $e_ie_j$ with $1\leq i<j\leq 8$ generate its even part $\C_0(\vf)$.
Let $Q_1$, $Q_2$ and $Q_3$ denote the $K$-subalgebras of $\C_0(\vf)$ generated by the pairs of elements $(e_2e_4,e_3e_4)$, $(e_6e_8,e_7e_8)$ and $(e_1e_2e_3e_4,e_1e_5)$, respectively.
By taking the squares of these generators, we see that 
\[Q_1\simeq (-c_2c_4,-c_3c_4)_K,\, Q_2\simeq (-c_6c_8,-c_7c_8)_K\text{ and }Q_3\simeq (c_1c_2c_3c_4,-c_1c_5)_K.\]
Note that $(e_ie_j)(e_ke_l)=(e_ke_l)(e_ie_j)$ when $i,j,k,l\in\{1,\dots,8\}$ are all distinct.
It follows that $xy=yx$ for any $x\in Q_i$, $y\in Q_j$ with $i,j\in\{1,2,3\}$ distinct.
We thus obtain a canonical $K$-algebra homomorphism $Q_1\otimes_K Q_2\otimes_K Q_3\to \C_0(\vf)$.
As the tensor product is simple, this is an embedding.
We denote by $A$ its image.

Since $\tau(e_ie_j)=-e_ie_j$ for $1\leq i<j\leq 8$, each of $Q_1$, $Q_2$ and $Q_3$ is stable under $\tau$ and such that $
\tau|_{Q_i}=\can_{Q_i}$ for $i\in\{1,2\}$.
We set $\s=\tau|_A$ and $\gamma=\tau|_{Q_3}$, and obtain that
\[(A,\s)= (Q_1,\can_{Q_1})\otimes (Q_2,\can_{Q_2})\otimes (Q_3,\gamma).\]
Since $\gamma(e_1e_2e_3e_4)=e_1e_2e_3e_4$ and $\gamma(e_1e_5)=-e_1e_5$, the involution $\gamma$ on $Q_3$ is orthogonal with  $\disc(\gamma)=-c_1c_5\sq{K}$.
We set $Z=K(e_1\cdots e_8)$.
Since $(e_1\dots e_8)^2=d\in\mg{K}$, we 
obtain that $Z$ is a quadratic $K$-subalgebra of $\C_0(\vf)$.
Since $\tau(e_1\cdots e_8)=e_1\cdots e_8$, we have  
$\tau|_Z=\id_Z$.
Since $e_1\cdots e_8$ commutes with $e_ie_j$ for any $i,j\in\{1,\dots,8\}$, we have that $Z$ is contained in the center of $\C_0(\vf)$.
Since $A\cap A(e_1\dots e_8)=0$, we obtain by comparing dimensions that $A\otimes_KZ=\C_0(\vf)$.
Therefore 
 \[(\C_0(\vf),\tau)\simeq (A\otimes_K Z,\s\otimes\id_Z).\]
If $d\notin\sq{K}$, then $Z\simeq K(\sqrt{d})$, and we obtain that 
\[(\C_0(\vf),\tau)\simeq (A,\s)_{K(\sqrt{d})}.\]
Assume now that $d\in\sq{K}$.
Then $Z\simeq K\times K$, 
and we conclude that 
\[(\C_0(\vf),\tau)\simeq (A,\s)\times (A,\s).\]
In particular, $\left((\End_K(V_\vf),\ad_\vf),(A,\s),(A,\s)\right)$ is a trialitarian triple, and this implies by \cite[Prop.~42.5]{KMRT98} that 
$\mathbf{PSim}^+(A,\s)\simeq \mathbf{PSim}^+(\vf)$.
\end{proof}

In \cite[Theorems 2 \& 3]{Gil97}, two different constructions are presented to produce $8$-dimensional quadratic forms $\vf$ of trivial discriminant for which $\mathbf{PSim}^+(\vf)$ is not $R$-trivial.
In view of \Cref{8dim-even-Clifford-computation}, those immediately provide examples of degree-$8$ totally decomposable $K$-algebras with orthogonal involution $(A,\s)$ for which 
$\mathbf{PSim}^+(A,\s)$ is not $R$-trivial.
Both constructions, however, will produce such examples where $\ind A=8$.
Using a construction from \cite[Theorem~7.5]{AB25}, we can nevertheless produce such examples where $\ind A=4$.

\begin{thm}\label{T:ind-4-non-R-trivial-ex}
    Let $k$ be a field with $\car(k)\neq 2$ and $a_0,a_1,a_2\in\mg{k}$ be such that $[k(\sqrt{a_0},\sqrt{a_1},\sqrt{a_2}):k]=8$.
    Let $K=k(t_1,t_2)$, the rational function field in two variables.
    Let $Q_1=(t_1,t_2)_K$ and $Q_2=(a_1t_1,a_2t_2)_K$.
    Then $$(A,\s)=\Ad\lla a_0\rra\otimes (Q_1,\can_{Q_1})\otimes (Q_2,\can_{Q_2})$$ is a totally decomposable $K$-algebra with orthogonal involution of degree $8$ with $\ind A=4$ 
    such that
    $\mathbf{PSim}^+(A,\s)$ is not $R$-trivial.
\end{thm}
\begin{proof}
Clearly $(A,\s)$ is totally decomposable and $\ind A\leq 4$.
Once we confirm that $\mathbf{PSim}^+(A,\s)$ is not $R$-trivial, it will follow by \Cref{main} that $\ind A=4$.

Consider the quadratic form $\vf=\lla t_1, t_2 \rra \perp -a_0\lla a_1t_1, a_2t_2 \rra$ over $K$.
We denote by $\tau$ the canonical involution on $\C_0(\vf)$.
Since $\vf$ has trivial discriminant, we obtain by \Cref{8dim-even-Clifford-computation} that 
\[(\C_0(\vf),\tau)= (A',\s')\times (A',\s')\quad\text{ and }\quad \mathbf{PSim}^+(\vf)=\mathbf{PSim}^+(A',\s')\]
for $(A',\s')=(Q_1,\can_{Q_1})\otimes (Q_2,\can_{Q_2})\otimes (Q_3,\gamma)$, where 
 $Q_3=(1,a_0)_K$ and $\gamma$ is an orthogonal involution on $Q_3$ with $\disc(\gamma)=a_0\sq{K}$.
Note that $Q_3$ is split and $(Q_3,\gamma)\simeq \Ad\lla a_0\rra$.
Therefore $(A',\s')\simeq (A,\s)$.
 This shows that
$\mathbf{PSim}^+(A,\s)\simeq \mathbf{PSim}^+(\vf)$.
By \cite[Theorem~7.5]{AB25}, $\mathbf{PSim}^+(\vf)$ is not $R$-trivial, hence neither is $\mathbf{PSim}^+(A,\s)$.
\end{proof}

\begin{cor}\label{C:ind4counterex} 
Let $K=k_0(X,Y,Z)$, the rational function field in three variables over a field $k_0$ of characteristic different from $2$.
Let $Q_1=(Y,Z)_K$,  $Q_2=((X-1)Y,(X+1)Z)_K$ and
$$(A,\s)=\Ad\lla X\rra\otimes (Q_1,\can_{Q_1})\otimes (Q_2,\can_{Q_2}).$$
Then $\mathbf{PSim}^+(A,\s)$ is not $R$-trivial, and there exists a field extension $F/K$ such that the natural map 
$\mathbf{PSim}^+(A,\s)(K)/R\to \mathbf{PSim}^+(A,\s)(F)/R$
is not surjective.
\end{cor}
\begin{proof}
Let $k'/k_0$ be an algebraic closure of $k_0$ and $K'=k'(X,Y,Z)$.
By Tsen-Lang Theory, every $9$-dimensional quadratic form over $K'$ is isotropic; see e.g.~\cite[Chap.~2, Cor.~15.3]{Scha85}.
In view of \cite[Theorem 7.5]{AB26}, this implies
that $\mathbf{PSim}^+(A,\s)(K')/R=\{1\}$.
On the other hand, 
it follows by \Cref{T:ind-4-non-R-trivial-ex} (applied over $k=k_0(X)$ with
$a_0=X$, $a_1=X-1$, $a_2=X+1$, 
$t_1=Y$ and $t_2=Z$) that $\mathbf{PSim}^+(A_{K'},\s_{K'})$ is not $R$-trivial.
Note that $\mathbf{PSim}^+(A_{K'},\s_{K'})(L)=\mathbf{PSim}^+(A,\s)(L)$ for any extension $L/K'$.
Hence there exists a field extension $F/K'$ such that 
$\mathbf{PSim}^+(A,\s)(F)/R\neq\{1\}$.
Since 
the natural map 
$\mathbf{PSim}^+(A,\s)(K)/R\to \mathbf{PSim}^+(A,\s)(F)/R$ factors through $\mathbf{PSim}^+(A,\s)(K')/R=\{1\}$, we conclude that this map is not surjective.
\end{proof}

The existence of these examples in degree $8$ and index $4$ leads us to speculate about a stronger statement on the failure of $R$-triviality.

\begin{qu}
    Let $(A,\s)$ be a $K$-algebra with orthogonal involution of trivial discriminant and such that $\deg A$ is a multiple of $8$ and $\ind A\geq 4$.
    Assume that $\s$ is not hyperbolic. Does it follow that $\mathbf{PSim}^+(A,\s)$ is not $R$-trivial?
\end{qu}

The tensor product of algebras with involution does not preserve $R$-triviality of the groups of proper projective similitudes.
More precisely, 
given a nonsplit $K$-quaternion algebra with orthogonal involution $(C,\gamma)$, it follows from \cite[Cor.~8]{BMT04} that, over some field extension $K'/K$, there exists a $K'$-algebra with orthogonal involution $(B,\tau)$ of degree $4$ such that $\mathbf{PSim}^+(C_{K'}\otimes B,\gamma_{K'}\otimes \tau)$ is not $R$-trivial. 
In that construction, $C$ is nonsplit and $\tau$ has nontrivial discriminant.
We obtain the following variation.

\begin{cor}
Let $(C,\gamma)$ be a $K$-algebra with orthogonal involution of even degree and with nontrivial discriminant.
Let $K'=K(X,Y,Z)$, $Q_1=(Y,Z)_{K'}$, $Q_2=(XY,(X+1)Z)_{K'}$ and $(B,\tau)=(Q_1,\can_{Q_1})\otimes (Q_2,\can_{Q_2})$.
Then the group
$\mathbf{PSim}^+(C\otimes_K B, \gamma\otimes\tau)$ is not $R$-trivial.
\end{cor}
\begin{proof}
Set $(A,\s)=(C\otimes_K B,\gamma\otimes \tau)$ and fix $d\in\mg{K}$ with $\disc(\gamma)=d\sq{K}$.
Then $d\notin\sq{K}$.
We may choose a field extension $L/K$ such that $d\notin\sq{L}$, $C_L$ is split and $(C_L,\gamma_L)\simeq\Ad(\la 1,-d\ra\perp \eta)$ for a hyperbolic quadratic form $\eta$ over $L$.
Since $d\notin\sq{L}$, we have $[L(X)(\sqrt{d},\sqrt{X},\sqrt{X+1}):L(X)]=8$.
Let $L'=L(X,Y,Z)$. Then 
$(A',\s')=\Ad\lla d\rra_{L'}\otimes (B_{L'},\tau_{L'})$.
Since $(C_{L'},\gamma_{L'})\simeq \Ad(\lla d\rra_{L'}\perp \eta_{L'})$ and $\eta_{L'}$ is hyperbolic, we obtain that the $L'$-algebras with involution 
$(C_{L'},\gamma_{L'})$ and $\Ad\lla d\rra_{L'}$ are Witt equivalent, and hence so are 
$(A_{L'},\s_{L'})$ and $(A',\s')$. 
In particular, for every field extension $L''/L'$, we have 
\[\mathbf{PSim}^+(A_{L'},\s_{L'})(L'')/R\simeq \mathbf{PSim}^+(A',\s')(L'')/R.\]
By \Cref{T:ind-4-non-R-trivial-ex}, applied with $a_0=d,a_1=X$ and $a_2=X+1$, we have that $\mathbf{PSim}^+(A',\s')$ is not $R$-trivial. 
From this, we conclude that $\mathbf{PSim}^+(A,\s)$ is not $R$-trivial.
\end{proof}

This observation leads us to put forward the following question.

\begin{qu}
Let $(C,\gamma)$ be a nonhyperbolic $K$-algebra with involution of even degree.
Does there always exist a field extension $K'/K$ and a $K'$-algebra with involution $(B,\tau)$ such that $\mathbf{PSim}^+(B,\tau)$ is $R$-trivial and $\mathbf{PSim}^+(C\otimes B,\gamma\otimes\tau)$ is not $R$-trivial?
\end{qu}

\subsection*{Acknowledgments}
We acknowledge funding by the \emph{Bijzonder Onderzoeksfonds, University of Antwerp} (project \emph{BOF Opvang MSCA IF}, ID 51418).

\section*{Declarations}

\subsection*{Data availability}
There is no associated data, which is not already contained in the text.

\subsection*{Conflict of interest} The authors declare that there is no conflict of interest.

\bibliographystyle{plain}

\end{document}